\documentclass[12pt, a4paper]{amsart}

\usepackage{amssymb, amscd, amsthm}

\newtheorem{Thm}{Theorem}[section]
\newtheorem{Prop}[Thm]{Proposition}

\newtheorem{Lem}[Thm]{Lemma}
\theoremstyle{remark}
\newtheorem{Rmk}[Thm]{Remark}
\newtheorem*{Notation}{Notation}
\newtheorem*{Ack}{Acknowledgement}

\newcommand{\FonP}[1]{$(\mathrm{P}_{#1})$}
\newcommand{\FonPtr}[1]{$(\mathrm{P}_{#1}^{\mathrm{tr}})$}
\newcommand{\Order}{\mathcal{O}}
\newcommand{\ideal}{\mathfrak{a}}
\newcommand{\maxid}{\mathfrak{p}}
\newcommand{\into}{\hookrightarrow}
\newcommand{\onto}{\twoheadrightarrow}
\newcommand{\isomto}{\overset{\sim}{\to}}
\newcommand{\compose}{\circ}
\newcommand{\tensor}{\otimes}
\newcommand{\loccit}{loc.\ cit.}
\newcommand{\algcl}[1]{\overline{#1}}
\newcommand{\Z}{\mathbb{Z}}
\newcommand{\Q}{\mathbb{Q}}
\newcommand{\ab}{\mathrm{ab}}
\newcommand{\sep}{\mathrm{sep}}
\newcommand{\ur}{\mathrm{ur}}
\newcommand{\alg}[1]{\mathbf{#1}}
\newcommand{\pfpqc}[1]{(\mathrm{Perf} / #1)_{\mathrm{fpqc}}}
\DeclareMathOperator{\Gal}{Gal}
\DeclareMathOperator{\Hom}{Hom}

\DeclareMathOperator{\Ker}{Ker}

\DeclareMathOperator{\Ext}{Ext}

\title[Fontaine's property (P$_{m}$)]
	{Fontaine's property (P$_{m}$) \\ at the maximal ramification break}

\author{Takashi Suzuki}
\address{
	Department of Mathematics, University of Chicago,
	5734 S University Ave,
	Chicago, IL 60637, USA
}
\email{suzuki@math.uchicago.edu}

\author{Manabu Yoshida}
\address{
	Kyushu Sangyo High School,
	Murasaki, Chikushino, Fukuoka, 818-8585, Japan
}
\email{manabuyoshida3@gmail.com}
\thanks{The second author is supported
in part by JSPS Core-to-Core Program 18005}

\date{September 19, 2013}

\begin{document}

\begin{abstract}
	We completely determine which extension of local fields
	satisfies Fontaine's property (P$_{m}$) for a given real number $m$.
	A key ingredient of the proof is the local class field theory of Serre and Hazewinkel.
\end{abstract}

\maketitle


\section{Introduction}
Let $K$ be a complete discrete valuation field with perfect residue field $k$.
All algebraic extensions of $K$ are taken inside a fixed algebraic closure of $K$.
We denote by $v_{K}$ the normalized valuation of $K$
and its extension to any algebraic extension of $K$.
For an algebraic extension $E / K$ and a non-negative real number $m$,
we denote by $\Order_{E}$ the ring of integers of $E$
and by $\ideal_{E / K}^{m}$ the set of elements
$x \in \Order_{E}$ with $v_{K}(x) \ge m$.
In \cite{Fon85}, Fontaine gave upper bounds for
ramification of Galois representations of $K$
appearing in the generic fibers of finite flat group schemes over $\Order_{K}$.
For this, he considered the following property \FonP{m}
for a finite Galois extension $L / K$ for each non-negative real number $m$:
\begin{itemize}
	\item[\FonP{m}]
		If $E$ is an algebraic extension of $K$
		and if there exists an $\Order_{K}$-algebra homomorphism
		$\Order_{L} \to \Order_{E} / \ideal_{E / K}^{m}$,
		then $L \subset E$.
\end{itemize}
By results of Fontaine (\cite[Proposition 1.5]{Fon85}) and the second author (\cite[Theorem 5.2]{Yos10}),
every finite Galois extension $L / K$ satisfies \FonP{m} for $m > u_{L / K}$
and does not satisfy \FonP{m} for $m < u_{L / K}$,
where $u_{L / K}$ is the maximal upper ramification break for $L / K$
(see Notation below for the definition).
For $m = u_{L/ K}$, it is easy to see that
trivial $L / K$ satisfies \FonP{m}
and non-trivial tame $L / K$ does not satisfy \FonP{m}.
Here we call $L / K$ tame if $u_{L / K} \le 1$
and wild if $u_{L / K} > 1$.

In this paper, we determine which wild $L / K$ satisfies \FonP{m} for $m = u_{L / K}$.
The result is the following,
which shows that the answer to the question depends
(only) on the residue field of $K$.

\begin{Thm}
	\label{Thm:Main}
	Let $K$ be a complete discrete valuation field
	with perfect residue field $k$ of characteristic $p > 0$
	and let $L$ be a finite wild Galois extension of $K$.
	Then $L / K$ satisfies \FonP{m} for $m = u_{L / K}$
	if and only if any finite extension of $k$ has no Galois extension of degree $p$.
\end{Thm}

Even if there exists a finite extension of $k$ that has a Galois extension of degree $p$
(for example, if $k$ is finite),
the extension $L / K$ still satisfies a property
that is weaker than \FonP{m} (see Remark \ref{Rmk:GeneralPerfect}).

Although it seems to be possible to give an elementary proof of this theorem,
we prove it by using the local class field theory of Serre and Hazewinkel
(\cite{Ser61}, \cite[Appendice]{DG70}).
The authors believe that this usage of that theory is interesting.
The idea is simple, so let us show here
how the local class field theory of Serre and Hazewinkel works,
by giving a sketch of proof of Theorem \ref{Thm:Main}
for abelian $L / K$ with algebraically closed $k$.
To prove that such $L / K$ satisfies \FonP{m} for $m = u_{L / K}$,
let $E / K$ be a finite extension and
let $\eta \colon \Order_{L} \to \Order_{E} / \ideal_{E / K}^{m}$ be
an $\Order_{K}$-algebra homomorphism.
It can be shown that $\eta$ factors through the quotient
	\[
			\eta
		\colon
			\Order_{L} / \ideal_{L / K}^{m}
		\to
			\Order_{E} / \ideal_{E / K}^{m},
	\]
which makes $\Order_{E} / \ideal_{E / K}^{m}$
a finite free module over $\Order_{L} / \ideal_{L / K}^{m}$.
Hence we obtain the corresponding norm map
	\[
			N_{\eta}
		\colon
			(\Order_{E} / \ideal_{E / K}^{m})^{\times}
		\to
			(\Order_{L} / \ideal_{L / K}^{m})^{\times}
	\]
on the groups of units,
which commutes with the natural norm maps from both of the groups to
$(\Order_{K} / \ideal_{K / K}^{m})^{\times}$.
These groups of units can be viewed as quasi-algebraic groups over $k$
in the sense of Serre (\cite[\S 1]{Ser60}).
The norms maps are homomorphisms of quasi-algebraic groups.
Applying the fundamental group functor of Serre (\cite[\S 6.1]{Ser60})
to these quasi-algebraic groups
and using his local class field theory (\cite{Ser61}),
we see that the map $N_{\eta}$ induces a homomorphism
	\[
			\Gal(E \cap K^{\ab} / K)
		\to
			\Gal(L / K)
	\]
that commutes with the restriction maps from $\Gal(K^{\ab} / K)$ to both of the groups.
Here $K^{\ab}$ is the maximal abelian extension of $K$.
This implies that $L \subset E \cap K^{\ab} \subset E$.
Thus $L / K$ satisfies \FonP{m}.

If we argue similarly in the case $k$ is a general perfect field,
what we can naturally get is slightly weaker than Theorem \ref{Thm:Main}, namely
$L$ is contained in the composite of $E$ and some unramified extension of $K$
that we cannot control.
This is because the local class field theory of Hazewinkel describes
only the inertia subgroup of $\Gal(K^{\ab} / K)$.
To control this unramified extension,
the authors refined the theory in a way that
allows us to describe the whole group $\Gal(K^{\ab} / K)$ canonically
and wrote it in a separate paper \cite{SY12}.
Using this result, we can prove Theorem \ref{Thm:Main} in a way similar to the above.

The organization of this paper is as follows.
In Section \ref{Sect:RedAb}, we reduce the theorem to the case
where $L / K$ has only one jump in the ramification filtration of its Galois group.
In this case, $L / K$ is an abelian extension.
In Section \ref{Sect:RedTR}, we reduce the theorem
so that we need only totally ramified $E / K$
for the extensions appearing in the definition of \FonP{m}.
In Section \ref{Sect:LCFT},
we recall the notation and results of \cite{SY12}.
With this, in Section \ref{Sect:RedCase},
we prove Theorem \ref{Thm:Main} for the reduced case.

\begin{Ack}
	The authors would like to express their sincere gratitude
	to Professor Kato and Professor Taguchi for having helpful discussions.
	They would like to thank Preston Wake for reading a draft of the paper
	and giving comments on it.
\end{Ack}

\begin{Notation}
	For a complete discrete valuation field $K$ with perfect residue field $k$,
	we denote by $v_{K}$ the discrete valuation
	with $v_{K}(K^{\times}) = \Z$,
	by $\Order_{K}$ the ring of integers,
	by $\maxid_{K}$ the maximal ideal,
	by $U_{K}$ the group of units,
	and by $U_{K}^{n}$ the group of $n$-th principal units ($n \ge 0$).
	All algebraic extensions of $K$ are
	taken inside a fixed algebraic closure $\algcl{K}$ of $K$.
	We denote by $K^{\sep}$ ($\subset \algcl{K}$) the separable closure of $K$,
	by $K^{\ur}$ the maximal unramified extension of $K$
	and by $K^{\ab}$ the maximal abelian extension of $K$.
	The same notation and convention are applied to $k$.
	The absolute Galois group $\Gal(K^{\sep} / K)$ (resp.\ $\Gal(\algcl{k} / k)$)
	is denoted by $G_{K}$ (resp.\ $G_{k}$).
	We extend $v_{K}$ to $\algcl{K}$.
	Note that $v_{K}(\algcl{K}^{\times}) = \Q \supset \Z$.
	For an algebraic extension $E / K$,
	we denote by $\Order_{E}$ the ring of integers,
	by $\hat{E}$ the completion of $E$,
	and by $e_{E / K}$ the ramification index of $E / K$.
	We define
		$
				\ideal_{E / K}^{m}
			=
				\{
						x \in \Order_{E}
					\,|\,
						v_{K}(x) \ge m
				\}
		$
	for a non-negative real number $m$.
	If $E / K$ is totally ramified,
	an extension of the form $E K' / K'$ for an unramified extension $K' / K$
	will be called an unramified base change of $E / K$.
	For a finite Galois extension $L / K$,
	let $\Gal(L / K)_{i}$, $\Gal(L / K)^{u}$ (resp.\ $\varphi_{L / K}$, $\psi_{L / K}$)
	be the ramification groups (resp.\ the Herbrand functions)
	that are studied in \cite{Ser79}.
	Also, let $\Gal(L / K)_{(i)}$, $\Gal(L / K)^{(u)}$
	(resp.\ $\tilde{\varphi}_{L / K}$, $\tilde{\psi}_{L / K}$)
	be the ramification groups (resp.\ the Herbrand functions)
	that are defined in \cite{Fon85}.
	Their relations are as follows: for real numbers $i, u \ge -1$, we have
		\begin{gather*}
				\Gal(L / K)_{i} = \Gal(L / K)_{((i + 1) / e_{L / K})},
			\quad
				\Gal(L / K)^{u} = \Gal(L / K)^{(u + 1)},
			\\
				\varphi_{L / K}(i) = \tilde{\varphi}_{L / K}((i + 1) / e_{L / K}) - 1,
			\quad
				\psi_{L / K}(u) = e_{L / K} \tilde{\psi}_{L / K}(u + 1) - 1.
		\end{gather*}
	We define the maximal upper ramification break for $L / K$
	to be the maximal real number $u_{L / K}$
	with non-trivial $\Gal(L / K)^{(u_{L / K})}$.
	Set $i_{L / K} = \tilde{\psi}_{L / K}(u_{L / K})$;
	it is the maximal real number with non-trivial $\Gal(L / K)_{(i_{L / K})}$.
\end{Notation}


\section{Reduction to abelian $L / K$}
\label{Sect:RedAb}
The following proposition allows us
to reduce the proof of Theorem \ref{Thm:Main}
to the case where $L / K$ has only one jump
in the ramification filtration of its Galois group.

\begin{Prop}
	\label{Prop:RedAb}
	Let $L / K$ be a finite Galois extension of complete discrete valuation fields
	with perfect residue fields.
	Let $M$ be the $\Gal(L / K)^{(u_{L / K})}$-fixed subfield of $L$.
	Then $L / K$ satisfies \FonP{m} for $m = u_{L / K}$
	if and only if $L / M$ satisfies \FonP{m} for $m = u_{L / M}$.
\end{Prop}

Note that any finite extension of the residue field of $K$
has no Galois extension of degree $p$
if and only if the same holds for $M$,
since the maximal pro-$p$ quotient of the absolute Galois group
of a field of characteristic $p$ is pro-$p$ free
(\cite[Chapter 2, \S 2.2, Corollary 1]{Ser02}).
Hence Proposition \ref{Prop:RedAb} allows us
to reduce the proof of Theorem \ref{Thm:Main} for $L / K$
to that for $L / M$.

The proposition is a direct consequence of the following lemma.

\begin{Lem}
	Let $K$, $L$, $M$ be as in Proposition \ref{Prop:RedAb}.
	Let $E$ be an algebraic extension of $K$.
	Then the following are equivalent:
	\begin{enumerate}
		\item \label{Enum:L/K}
			There exists an $\Order_{K}$-algebra homomorphism
			$\Order_{L} \to \Order_{E} / \ideal_{E / K}^{u_{L / K}}$.
		\item \label{Enum:L/M}
			The field $M$ is contained in $E$ and
			there exists an $\Order_{M}$-algebra homomorphism
			$\Order_{L} \to \Order_{E} / \ideal_{E / M}^{u_{L / M}}$.
	\end{enumerate}
\end{Lem}

\begin{proof}
	$(\ref{Enum:L/K}) \implies (\ref{Enum:L/M})$.
	First we show that $M$ is contained in $E$.
	The case $L / K$ is unramified is trivial,
	so we assume $L / K$ is not unramified.
	Then $u_{L / K} > u_{M / K}$.
	By assumption, there exists an $\Order_{K}$-algebra homomorphism
	$\eta \colon \Order_{L} \to \Order_{E} / \ideal_{E / K}^{u_{L / K}}$.
	Consider the composite of the inclusion $\Order_{M} \into \Order_{L}$ and $\eta$.
	The extension $M / K$ satisfies \FonP{m} for $m = u_{L / K} > u_{M / K}$.
	Hence $M$ is contained in $E$.
	
	Next we show the existence of an $\Order_{M}$-algebra homomorphism
	$\Order_{L} \to \Order_{E} / \ideal_{E / M}^{u_{L / M}}$.
	Let $\alpha$ be a generator of $\Order_{L}$ as an $\Order_{K}$-algebra,
	let $f$ be the minimal polynomial of $\alpha$ over $K$,
	and let $\beta$ be a lift of $\eta(\alpha) \in \Order_{E} / \ideal_{E / K}^{u_{L / K}}$
	to $\Order_{E}$.
	We have $v_{K}(f(\beta)) \ge u_{L / K}$
	by the well-definedness of $\eta$.
	By \cite[the second proposition of \S 1.3]{Fon85},
	we have
		$
				\tilde{\psi}_{L / K}(v_{K}(f(\beta)))
			=
				v_{K}(\beta - \sigma(\alpha))
		$
	for some $\sigma \in \Gal(L / K)$.
	Hence we have
		$
				v_{K}(\beta - \sigma(\alpha))
			\ge
				\tilde{\psi}_{L / K}(u_{L / K})
			=
				i_{L / K}
		$.
	Since
		$
				\Gal(L / M)_{(i)}
			=
				\Gal(L / K)_{(i / e_{M / K})} \cap \Gal(L / M)
		$
	for $i \ge 0$ and $\Gal(L / M) = \Gal(L / K)_{(i_{L / K})}$,
	we have $i_{L / M} = e_{M / K} i_{L / K}$.
	Hence $v_{M}(\beta - \sigma(\alpha)) \ge e_{M / K} i_{L / K} = i_{L / M}$.
	Thus we have $v_{M}(g(\beta)) \ge u_{L / M}$
	by \cite[\loccit]{Fon85},
	where $g$ is the minimal polynomial of $\sigma(\alpha)$ over $M$.
	Since
		$
				\Order_{L}
			=
				\Order_{K}[\sigma(\alpha)]
			=
				\Order_{M}[\sigma(\alpha)]
		$,
	we can define an $\Order_{M}$-algebra homomorphism
	$\Order_{L} \to \Order_{E} / \ideal_{E / K}^{u_{L / M}}$
	by sending $\sigma(\alpha)$ to $\beta$.
	This proves $(\ref{Enum:L/K}) \implies (\ref{Enum:L/M})$.
	
	$(\ref{Enum:L/M}) \implies (\ref{Enum:L/K})$.
	By assumption, there exists an $\Order_{M}$-algebra homomorphism
	$\eta \colon \Order_{L} \to \Order_{E} / \ideal_{E / M}^{u_{L / M}}$.
	Let $\alpha$ be a generator of $\Order_{L}$ as an $\Order_{K}$-algebra,
	let $g$ be the minimal polynomial of $\alpha$ over $M$,
	and let $\beta$ be a lift of $\eta(\alpha) \in \Order_{E} / \ideal_{E / M}^{u_{L / M}}$
	to $\Order_{E}$.
	We have $v_{M}(g(\beta)) \ge u_{L / M}$
	by the well-definedness of $\eta$.
	We have
		$
				\tilde{\psi}_{L / M}(v_{M}(g(\beta)))
			=
				v_{M}(\beta - \sigma(\alpha))
		$
	for some $\sigma \in \Gal(L / M)$ by \cite[\loccit]{Fon85}.
	Hence
		$
				v_{M}(\beta - \sigma(\alpha))
			\ge
				\tilde{\psi}_{L / M}(u_{L / M})
			=
				i_{L / M}
		$.
	Using the equality $i_{L / M} = e_{M / K} i_{L / K}$
	obtained above,
	we have $v_{K}(\beta - \sigma(\alpha)) \ge i_{L / K}$.
	Thus we have $v_{K}(f(\beta)) \ge u_{L / K}$
	by \cite[\loccit]{Fon85},
	where $f$ is the minimal polynomial of $\alpha$ over $K$.
	Thus we can define an $\Order_{K}$-algebra homomorphism
	$\Order_{L} \to \Order_{E} / \ideal_{E / K}^{u_{L / K}}$
	by sending $\alpha$ to $\beta$.
	This proves $(\ref{Enum:L/M}) \implies (\ref{Enum:L/K})$.
\end{proof}


\section{Reduction to totally ramified $E / K$}
\label{Sect:RedTR}
In this section, we reduce the proof of Theorem \ref{Thm:Main}
to the case that the extensions $E / K$
appearing in the definition of \FonP{m} are totally ramified.
To be more precise, we consider the following property \FonPtr{m}:
\begin{itemize}
	\item[\FonPtr{m}]
		If $E$ is a totally ramified algebraic extension of $K$
		and if there exists an $\Order_{K}$-algebra homomorphism
		$\Order_{L} \to \Order_{E} / \ideal_{E / K}^{m}$,
		then $L \subset E$.
\end{itemize}
\begin{Prop}
	\label{Prop:RedTR}
	Let $K$ be a complete discrete valuation field
	with perfect residue field $k$,
	let $L$ be a finite Galois totally ramified extension of $K$
	and let $m$ be a non-negative real number.
	Then $L / K$ satisfies \FonP{m} if and only if
	any finite unramified base change $L' / K'$ of $L / K$ satisfies \FonPtr{m}.
\end{Prop}

\begin{proof}
	First we assume that any finite unramified base change $L' / K'$ of $L / K$
	satisfies \FonPtr{m}
	and show that $L / K$ satisfies \FonP{m}.
	Suppose $E / K$ is an algebraic extension
	and there exists an $\Order_{K}$-algebra homomorphism
	$\eta \colon \Order_{L} \to \Order_{E} / \ideal_{E / K}^{m}$.
	We may assume $E / K$ is finite.
	Let $K' = E \cap K^{\ur}$ and set $L' = L K'$.
	Let $\eta' \colon \Order_{L'} \to \Order_{E} / \ideal_{E / K'}^{m}$
	be the base change of $\eta$ by $\Order_{K'}$.
	By assumption, we have $L' \subset E$.
	Hence $L \subset E$.
	This shows that $L / K$ satisfies \FonP{m}.
	
	Next we assume that $L / K$ satisfies \FonP{m}
	and show that any finite unramified base change $L' / K'$ of $L / K$
	satisfies \FonPtr{m}.
	Suppose $E / K'$ is a totally ramified algebraic extension
	and there exists an $\Order_{K'}$-algebra homomorphism
	$\eta' \colon \Order_{L'} \to \Order_{E} / \ideal_{E / K'}^{m}$.
	Consider the composite of the inclusion $\Order_{L} \into \Order_{L'}$
	and $\eta'$.
	By assumption, \FonP{m} is true for $L / K$.
	Hence $L \subset E$.
	Thus $L' = L K' \subset E$.
	This shows that $L' / K'$ satisfies \FonPtr{m}.
\end{proof}


\section{Preliminary results on the local class field theory of Serre and Hazewinkel}
\label{Sect:LCFT}
In the paper \cite{SY12},
the authors refined the local class field theory of Serre and Hazewinkel
allowing extensions that are not necessarily totally ramified.
In this section, we recall the notation and results of the cited paper without proofs.
We need this even after the reduction steps made in Sections \ref{Sect:RedAb} and \ref{Sect:RedTR}
since the composite of the totally ramified extensions $L$ and $E$ there
might not be totally ramified.

We recall the notation.
We work on the site $\pfpqc{k}$
of perfect schemes over a perfect field $k$ with the fpqc topology.
The category of sheaves of abelian groups on $\pfpqc{k}$ contains
both the category of commutative affine pro(-quasi)-algebraic groups over $k$
in the sense of Serre (\cite{Ser60})
and the category of commutative \'etale group schemes over $k$
as thick abelian full subcategories.
For $i \ge 0$, we denote by $\Ext_{k}^{i}$ the $i$-th Ext functor
for the category of sheaves of abelian groups on $\pfpqc{k}$.
For a sheaf $A$ of abelian groups on $\pfpqc{k}$ and a non-negative integer $i$,
we define the $i$-th homotopy group $\pi_{i}^{k}(A)$ of $A$
to be the Pontryagin dual of the injective limit of the torsion abelian groups
$\Ext_{k}^{i}(A, n^{-1} \Z / \Z)$ for $n \ge 1$.
The system $\{\pi_{i}^{k}\}_{i \ge 0}$
is a covariant homological functor
from the category of sheaves of abelian groups on $\pfpqc{k}$
to the category of profinite abelian groups.

Let $K$ be a complete discrete valuation field with perfect residue field $k$
and let $\Order_{K}$ be its ring of integers.
We define a sheaf $\alg{O}_{K}$ of rings on $\pfpqc{k}$ as follows.
For each perfect $k$-algebra $R$, we set
$\alg{O}_{K}(R) = W(R) \tensor_{W(k)} \Order_{K}$
if $K$ has mixed characteristic,
and $\alg{O}_{K}(R) = R \mathbin{\hat{\tensor}_{k}} \Order_{K}$
if $K$ has equal characteristic,
where $W$ is the sheaf of the rings of Witt vectors of infinite length
and $\mathbin{\hat{\tensor}}$ denotes the completed tensor product.
Let $\alg{K}$ be the sheaf of rings on $\pfpqc{k}$
with $\alg{K}(R) = \alg{O}_{K}(R)[(1 \tensor \pi_{K})^{-1}]$,
where $\pi_{K}$ is a prime element of $\Order_{K}$.
This is independent of the choice of $\pi_{K}$.
We set $\alg{U}_{K} = \alg{O}_{K}^{\times}$.
For each $n \ge 0$, the sheaf of rings $\alg{O}_{K}$ has
a subsheaf of ideals $\alg{p}_{K}^{n}$ with
	$
			\alg{p}_{K}^{n}(R)
		=
			\alg{O}_{K}(R) \tensor_{\Order_{K}} \maxid_{K}^{n}
	$
for a perfect $k$-algebra $R$.
The presentation
	$
			\alg{O}_{K}
		=
			\projlim_{n \to \infty}
				\alg{O}_{K} / \alg{p}_{K}^{n}
	$
gives an affine proalgebraic ring structure for $\alg{O}_{K}$.
Likewise, $\alg{U}_{K}$ has a subsheaf of groups
$\alg{U}_{K}^{n} = 1 + \alg{p}_{K}^{n}$ for each $n \ge 1$
(for $n = 0$, we set $\alg{U}_{K}^{0} = \alg{U}_{K}$).
The presentation
	$
			\alg{U}_{K}
		=
			\projlim_{n \to \infty}
				\alg{U}_{K} / \alg{U}_{K}^{n}
	$
gives an affine proalgebraic group structure for $\alg{U}_{K}$.
We have a split exact sequence
$0 \to \alg{U}_{K} \to \alg{K}^{\times} \to \Z \to 0$.
The same constructions are applied to
any finite totally ramified extension $E$ of $K$,
which yields sheaves $\alg{U}_{E}$, $\alg{U}_{E}^{n}$ and $\alg{E}^{\times}$ on $\pfpqc{k}$.

Theorem \ref{Thm:LCFT} and Propositions \ref{Prop:SeqFinExt}-\ref{Prop:CokNorm} below
correspond to Theorem 1.1 and Propositions 4.3-4.5 of \cite{SY12} respectively.
We state slightly restricted versions.
These are sufficient and suitable for this paper.
See \cite{SY12} for the proofs.

\begin{Thm}
	\label{Thm:LCFT}
	There exists a canonical isomorphism
	$\pi_{1}^{k}(\alg{K}^{\times}) \isomto G_{K}^{\ab}$
	and a commutative diagram with exact rows
		\begin{equation}
			\label{Eq:LCFT}
			\begin{CD}
					0
				@>>>
					\pi_{1}^{k}(\alg{U}_{K})
				@>>>
					\pi_{1}^{k}(\alg{K}^{\times})
				@>>>
					\pi_{1}^{k}(\Z)
				@>>>
					0
				\\
				@.
				@VV \wr V
				@VV \wr V
				@VV \wr V
				@.
				\\
					0
				@>>>
					T(K^{\ab} / K)
				@>>>
					G_{K}^{\ab}
				@>>>
					G_{k}^{\ab}
				@>>>
					0.
			\end{CD}
		\end{equation}
	Here $G_{K}^{\ab}$ (resp.\ $G_{k}^{\ab}$) is the Galois group
	of the maximal abelian extension $K^{\ab} / K$ (resp.\ $k^{\ab} / k$)
	and $T$ denotes the inertia group.
	The left vertical isomorphism is the reciprocity map
	of the local class field theory of Hazewinkel
	(\cite[Appendice, \S 7.3]{DG70}).
	The right vertical isomorphism is dual to the isomorphism
	$H^{1}(G_{k}, \Q / \Z) \cong \Ext_{k}^{1}(\Z, \Q / \Z)$ times $-1$.
\end{Thm}

\begin{Prop}
	\label{Prop:SeqFinExt}
	Let $E / K$ be a finite totally ramified extension
	and let $N_{E / K} \colon \alg{E}^{\times} \to \alg{K}^{\times}$
	be the norm map.
	Then we have
		$
				\pi_{0}^{k}(\Ker(N_{E / K}))
			\cong
				\Gal(E \cap K^{\mathrm{ab}} / K)
		$.
\end{Prop}

For the next two propositions,
let $L / K$ be a finite wild Galois extension
with a unique jump in the ramification filtration of its Galois group.
Then $L / K$ is a totally ramified abelian extension
whose Galois group is killed by $p$.
Set $m = u_{L / K} > 1$
and $G = \Gal(L / K)$.
Then $m$ is an integer by the Hasse-Arf theorem.
We write $\psi = \psi_{L / K}$.

\begin{Prop}
	\label{Prop:SeqRamFil}
	Let $N_{L / K} \colon \alg{U}_{L} \to \alg{U}_{K}$ be the norm map
	and
		\[
				\overline{N_{L / K}}
			\colon
				\alg{U}_{L} / \alg{U}_{L}^{\psi(m - 1) + 1}
			\to
				\alg{U}_{K} / \alg{U}_{K}^{m}
		\]
	be its quotient.
	Then we have
	$\pi_{0}^{k}(\overline{N_{L / K}}) \cong G$.
\end{Prop}

\begin{Prop}
	\label{Prop:CokNorm}
	We have canonical isomorphisms
		\[
				K^{\times} / N_{L / K} L^{\times}
			\cong
				U_{K}^{m - 1} / U_{K}^{m} N_{L / K} U_{L}^{\psi(m - 1)}
			\cong
				\Hom(G_{k}, G).
		\]
	Moreover, if $k'$ is a finite extension of $k$
	and $K'$ (resp.\ $L'$) is the corresponding
	unramified base change of $K$ (resp.\ $L$),
	then the following diagram commutes:
		\[
			\begin{CD}
					K^{\times} / N_{L / K} L^{\times}
				@=
					U_{K}^{m - 1} / U_{K}^{m} N_{L / K} U_{L}^{\psi(m - 1)}
				@=
					\Hom(G_{k}, G)
				\\
				@VVV
				@VVV
				@VVV
				\\
					K'^{\times} / N_{L' / K'} L'^{\times}
				@=
					U_{K'}^{m - 1} / U_{K'}^{m} N_{L' / K'} U_{L'}^{\psi(m - 1)}
				@=
					\Hom(G_{k'}, G),
			\end{CD}
		\]
	where the vertical arrows
	are induced by the inclusions $K^{\times} \into K'^{\times}$,
	$U_{K}^{m - 1} \into U_{K'}^{m - 1}$ and
	$G_{k'} \into G_{k}$.
\end{Prop}


\section{The reduced case}
\label{Sect:RedCase}
Throughout this section,
let $K$ be a complete discrete valuation field
with perfect residue field $k$ of characteristic $p > 0$
and let $L$ be a finite wild Galois extension of $K$
with a unique jump in the ramification filtration of its Galois group.
Then the extension $L / K$ is a totally ramified abelian extension
whose Galois group is killed by $p$.
We put $G := \Gal(L / K)$ and $m := u_{L / K} > 1$.
By the Hasse-Arf theorem, $m$ is an integer.

The following proposition, combined with
Propositions \ref{Prop:RedAb} and \ref{Prop:RedTR},
proves Theorem \ref{Thm:Main}.

\begin{Prop}
	\label{Prop:RedMainTh}
	Let $L / K$, $G$, $m$ be as above.
	\begin{enumerate}
		\item
			\label{Ass:OnlyIf}
			If $L / K$ satisfies \FonPtr{m},
			then $k$ has no Galois extension of degree $p$.
		\item
			\label{Ass:If}
			If $E$ is a totally ramified algebraic extension of $K$
			and if there exists an $\Order_{K}$-algebra homomorphism
			$\Order_{L} \to \Order_{E} / \ideal_{E / K}^{m}$,
			then there exists an unramified Galois extension $K' / K$
			with Galois group isomorphic to a subgroup of $G$
			such that $L \subset E K'$.
	\end{enumerate}
\end{Prop}

\begin{proof}
	(\ref{Ass:OnlyIf}).
	Assume $L / K$ satisfies \FonPtr{m}.
	Since $G = \Gal(L / K)$ is a non-trivial abelian group killed by $p$,
	to show that $k$ has no Galois extension of degree $p$,
	it is enough to prove that $\Hom(G_{k}, G) = 0$,
	which is equivalent to showing that
	$N_{L / K} U_{L}^{\psi_{L / K}(m - 1)} = U_{K}^{m - 1}$
	by Proposition \ref{Prop:CokNorm}.
	Take $u \in U_{K}^{m - 1}$.
	Let $f(x) = x^{n} + a_{n - 1} x^{n - 1} + \dots + a_{1} x + a_{0}$
	be the minimal polynomial of a prime element $\pi_{L}$ of $\Order_{L}$.
	Then the modified polynomial
	$x^{n} + a_{n - 1} x^{n - 1} + \dots + a_{1} x + u a_{0} \in \Order_{K}[x]$
	is still an Eisenstein polynomial.
	Let $\pi_{E}$ be a root of this polynomial such that
	the function $v_{K}(\pi_{E} - \sigma \pi_{L})$ for $\sigma \in G$
	takes the maximum at $\sigma = 1$.
	Set $E = K(\pi_{E})$.
	Then we have a well-defined $\Order_{K}$-algebra homomorphism
	$\Order_{L} \to \Order_{E} / \ideal_{E / K}^{m}$
	that sends $\pi_{L}$ to $\pi_{E}$,
	since $v_{K}(f(\pi_{E})) = v_{K}((u - 1) a_{0}) \ge m$.
	Since $L / K$ satisfies \FonPtr{m} by assumption,
	we have $L \subset E$, which is actually an equality: $L = E$.
	Since $\pi_{L}$ and $\pi_{E}$ are both prime in $\Order_{L} = \Order_{E}$,
	their ratio $u' := \pi_{E} / \pi_{L}$ is in $U_{L} = U_{E}$.
	By \cite[the second proposition of \S 1.3]{Fon85},
	we have
		$
				\tilde{\psi}_{L / K}(v_{K}(f(\pi_{E})))
			=
				v_{K}(\pi_{E} - \pi_{L})
			=
				(v_{L}(u' - 1) + 1) / e_{L / K}
		$.
	Thus we have
		$
				v_{L}(u' - 1)
			\ge
				e_{L / K} \tilde{\psi}_{L / K}(m) - 1
			=
				\psi_{L / K}(m - 1)
		$.
	Hence $u' \in U_{L}^{\psi_{L / K}(m - 1)}$.
	Also we have
		$
				N_{L / K} u'
			=
				N_{L / K} \pi_{E} / N_{L / K} \pi_{L}
			=
				u a_{0} / a_{0}
			=
				u
		$.
	Hence $N_{L / K} U_{L}^{\psi_{L / K}(m - 1)} = U_{K}^{m - 1}$.
	This proves (\ref{Ass:OnlyIf}).
	
	(\ref{Ass:If}).
	Let $E$ be a totally ramified algebraic extension of $K$
	such that there exists an $\Order_{K}$-algebra homomorphism
	$\eta \colon \Order_{L} \to \Order_{E} / \ideal_{E / K}^{m}$.
	We may assume $E / K$ is finite.
	We regard $\Order_{E} / \ideal_{E / K}^{m}$
	as an $\Order_{L}$-module via $\eta$.
	By Lemma \ref{Lem:FiniteFree} below,
	we know that $\Order_{E} / \ideal_{E / K}^{m}$ is finite free
	over $\Order_{L} / \ideal_{L / K}^{m}$.
	Hence we get the corresponding norm map $N_{\eta}$
	from
		$
				(\Order_{E} / \ideal_{E / K}^{m})^{\times}
			=
				U_{E} / U_{E}^{m e_{E / K}}
		$
	to
		$
				(\Order_{L} / \ideal_{L / K}^{m})^{\times}
			=
				U_{L} / U_{L}^{m e_{L / K}}
		$.
	The map $N_{\eta}$ can be extended to a morphism
		$
				\alg{U}_{E} / \alg{U}_{E}^{m e_{E / K}}
			\to
				\alg{U}_{L} / \alg{U}_{L}^{m e_{L / K}}
		$
	of sheaves of abelian groups on $\pfpqc{k}$.
	By the transitivity of norm maps,
	we have $\overline{N_{L / K}} \compose N_{\eta} = \overline{N_{E / K}}$,
	where
		$
				\overline{N_{L / K}}
			\colon
				\alg{U}_{L} / \alg{U}_{L}^{m e_{L / K}}
			\to
				\alg{U}_{K} / \alg{U}_{K}^{m}
		$
	(resp.\
		$
				\overline{N_{E / K}}
			\colon
				\alg{U}_{E} / \alg{U}_{E}^{m e_{E / K}}
			\to
				\alg{U}_{K} / \alg{U}_{K}^{m}
		$)
	is the map induced by the norm map $N_{L / K}$
	for the finite extension $L / K$
	(resp.\ $N_{E / K}$ for $E / K$).
	Since $m e_{L / K} \ge \psi_{L / K}(m - 1) + 1$,
	the map $\overline{N_{L / K}}$ factors through the canonical projection
		$
				\alg{U}_{L} / \alg{U}_{L}^{m e_{L / K}}
			\onto
				\alg{U}_{L} / \alg{U}_{L}^{\psi_{L / K}(m - 1) + 1}
		$.
	By abuse of notation, we write $\overline{N_{L / K}}$
	for the map
		$
				\alg{U}_{L} / \alg{U}_{L}^{\psi_{L / K}(m - 1) + 1}
			\to
				\alg{U}_{K} / \alg{U}_{K}^{m}
		$.
	
	We show that there exists an unramified Galois extension $K' / K$
	with Galois group isomorphic to a subgroup of $G$
	such that the following holds:
	Let $k'$ be the residue field of $K'$
	and let $L' = L K'$, $E' = E K'$ and $N_{\eta'}$
	be the unramified base changes by $K' / K$.
	Then we can extend $N_{\eta'}$ to a morphism $\tilde{N}_{\eta'}$
	such that we have the following commutative diagram with exact rows
	of sheaves of abelian groups on $\pfpqc{k'}$:
		\[
			\begin{CD}
					0
				@>>>
					\Ker(N_{E' / K'})
				@>>>
					\alg{E}'^{\times}
				@>> N_{E' / K'} >
					\alg{K}'^{\times}
				@>>>
					0
				\\
				@.
				@VV \tilde{N}_{\eta'} V
				@VV \tilde{N}_{\eta'} V
				@VVV
				@.
				\\
					0
				@>>>
					\Ker(\overline{N_{L' / K'}})
				@>>>
					\alg{L}'^{\times} / \alg{U}_{L'}^{\psi_{L / K}(m - 1) + 1}
				@> \overline{N_{L' / K'}} >>
					\alg{K}'^{\times} / \alg{U}_{K'}^{m}
				@>>>
					0.
			\end{CD}
		\]
	The construction of $\tilde{N}_{\eta'}$ is as follows.
	Take a prime element $\pi_{E}$ of $\Order_{E}$.
	By the isomorphism of Proposition \ref{Prop:CokNorm},
	$N_{E / K} \pi_{E}$ defines a homomorphism $\chi \colon G_{k} \to G$.
	Let $k'$ correspond to $\Ker(\chi)$ via Galois theory.
	Let $K'$, $L'$, $E'$, $N_{\eta'}$ be the corresponding unramified base changes.
	The restriction $\chi|_{G_{k'}} \colon G_{k'} \to G$ is a trivial character.
	Hence the image of $N_{E / K} \pi_{E}$ in $K'^{\times} / N_{L' / K'} L'^{\times}$
	is trivial by Proposition \ref{Prop:CokNorm}.
	Thus $N_{E / K} \pi_{E}$ can be written as $N_{L' / K'} \pi_{L'}$
	for some prime $\pi_{L'}$ of $\Order_{L'}$.
	We define $\tilde{N}_{\eta'} \pi_{E} := \pi_{L'}$.
	Then the commutativity of the right square
	follows from the equality
	$\overline{N_{L / K}} \compose N_{\eta} = \overline{N_{E / K}}$
	as a map from $\alg{U}_{E}$ to $\alg{U}_{K} / \alg{U}_{K}^{m}$
	and the equality
		$
				N_{E' / K'} \pi_{E}
			=
				N_{L' / K'} \pi_{L'}
		$.
	
	The above diagram and
	the resulting long exact sequences of homotopy groups give
	the following commutative diagram:
		\[
			\begin{CD}
					\pi_{1}^{k'}(\alg{K}'^{\times})
				@>>>
					\pi_{0}^{k'}(\Ker(N_{E' / K'}))
				\\
				@VVV
				@VV \tilde{N}_{\eta'} V
				\\
					\pi_{1}^{k'}(\alg{K}'^{\times} / \alg{U}_{K'}^{m})
				@>>>
					\pi_{0}^{k'}(\Ker(\overline{N_{L' / K'}})).
			\end{CD}
		\]
	By Theorem \ref{Thm:LCFT}, Proposition \ref{Prop:SeqFinExt}
	and Proposition \ref{Prop:SeqRamFil},
	we can translate this diagram into%
	\footnote{For the left bottom corner,
	use the facts $\pi_{1}^{k'}(\alg{U}_{K'}^{m}) \cong (G_{K'}^{\ab})^{m}$
	and $\pi_{0}^{k'}(\alg{U}_{K'}^{m}) = 0$.
	The first fact is \cite[\S 3.2, Th.\ 1]{Ser61}
	and the second from the connectedness of $\alg{U}_{K'}^{m}$.}
		\[
			\begin{CD}
					G_{K'}^{\ab}
				@>>>
					\Gal(E' \cap K'^{\ab} / K')
				\\
				@VVV
				@VVV
				\\
					G_{K'}^{\ab} / (G_{K'}^{\ab})^{m}
				@>>>
					G = \Gal(L' / K').
			\end{CD}
		\]
	The horizontal arrows are restriction maps.
	The commutativity of this diagram
	shows that $L' \subset E' \cap K'^{\ab}$.
	Thus $L \subset E' = E K'$.
	This proves (\ref{Ass:If}).
\end{proof}

\begin{Lem}
	\label{Lem:FiniteFree}
	The $\Order_{L}$-module $\Order_{E} / \ideal_{E / K}^{m}$
	is killed by $\ideal_{L / K}^{m}$
	and is finite free over $\Order_{L} / \ideal_{L / K}^{m}$.
\end{Lem}

\begin{proof}
	(Another proof can be found in \cite[Lem.\ 2.1]{HT08}.)
	Put $n = [L : K] = e_{L / K}$ and $T = \Order_{E} / \ideal_{E / K}^{m}$.
	Let $\pi_{K}$ be a prime element of $\Order_{K}$,
	let $\pi_{L}$ be a prime element of $\Order_{L}$
	and let $\eta \in \Order_{E}$ be a lift of
	$\eta(\pi_{L}) \in \Order_{E} / \ideal_{E / K}^{m}$.
	Since $u = \pi_{L}^{n} / \pi_{K}$ is a unit,
	$\eta(u)$ is a unit in $\Order_{E} / \ideal_{E / K}^{m}$.
	Since $\beta^{n} \equiv \eta(u) \pi_{K}$ in $\Order_{E} / \ideal_{E / K}^{m}$
	and $m > 1$, we have $v_{K}(\beta^{n}) = v_{K}(\pi_{K}) = 1$.
	Hence $v_{K}(\beta) = 1 / n$.
	Thus $\pi_{L}^{n m}$ kills $T = \Order_{E} / \ideal_{E / K}^{m}$.
	For an arbitrary element $\gamma \in \Order_{E}$
	that is a lift of an element of the $\pi_{L}$-torsion part $T[\pi_{L}]$ of $T$,
	we have $v_{K}(\beta \gamma) \ge m$.
	Hence $v_{K}(\gamma) \ge m - 1 / n = v_{K}(\beta^{n m - 1})$.
	Hence $\gamma \in \beta^{n m - 1} \Order_{E}$.
	Thus $T[\pi_{L}] = \pi_{L}^{n m - 1} T$.
	Applying the structure theorem of modules over a principal ideal domain,
	we get the result.
\end{proof}

\begin{Rmk}
	\label{Rmk:GeneralPerfect}
	One can use Propositions \ref{Prop:RedAb} and \ref{Prop:RedTR}
	to get a generalization of assertion (\ref{Ass:If})
	of Proposition \ref{Prop:RedMainTh}
	for a finite wild Galois extension $L / K$.
	The result is the following.
	If $E$ is an algebraic extension of $K$
	and if there exists an $\Order_{K}$-algebra homomorphism
	$\Order_{L} \to \Order_{E} / \ideal_{E / K}^{m}$ for $m = u_{L / K}$,
	then the $\Gal(L / K)^{(m)}$-fixed subfield $M$ of $L$ is contained in $E$,
	and there exists a finite unramified subextension $M' / M$ of $E / M$
	and an unramified Galois extension $M'' / M'$ with $\Gal(M'' / M')$
	isomorphic to a subgroup of $\Gal(L / K)^{(m)}$
	such that $L \subset E M''$.
\end{Rmk}

\end{document}